\documentclass[11pt,reqno]{amsart}
\usepackage{amsmath}
\usepackage{amsthm}
\usepackage{latexsym}
\usepackage{graphicx}
\usepackage{amssymb}
\usepackage[latin1]{inputenc}

\theoremstyle{plain}

\numberwithin{equation}{section}
\newtheorem{thm}{Theorem}[section]
\newtheorem{theorem}[thm]{Theorem}
\newtheorem{example}[thm]{Example}
\newtheorem{lemma}[thm]{Lemma}

\newtheorem{remark}[thm]{Remark}

\begin{document}

\setcounter{page}{1}
\title[Periodic representations for cubic irrationalities]{Periodic representations for cubic irrationalities}
\author{Marco Abrate}
\address{Dipartimento di Matematica\\
                Universit\`{a} di Torino\\
                Via Carlo Alberto 8\\
                Torino, Italy}
\email{marco.abrate@unito.it}
\author{Stefano Barbero}
\address{Dipartimento di Matematica\\
                Universit\`{a} di Torino\\
                Via Carlo Alberto 8\\
                Torino, Italy}
\email{stefano.barbero@unito.it}
\author{Umberto Cerruti}
\address{Dipartimento di Matematica\\
                Universit\`{a} di Torino\\
                Via Carlo Alberto 8\\
                Torino, Italy}
\email{umberto.cerruti@unito.it}
 \author{Nadir Murru}
\address{Dipartimento di Matematica\\
                Universit\`{a} di Torino\\
                Via Carlo Alberto 8\\
                Torino, Italy}
\email{nadir.murru@unito.it}

\begin{abstract}
In this paper we present some results related to the problem of finding periodic representations for algebraic numbers. In particular, we analyze the problem for cubic irrationalities. We show an interesting relationship between the convergents of bifurcating continued fractions related to a couple of cubic irrationalities, and a particular generalization of the R\'edei polynomials. Moreover, we give a method to construct a periodic bifurcating continued fraction for any cubic root paired with another determined cubic root.
\end{abstract}

\maketitle

\section{Introduction}

In 1839 Hermite \cite{Her} posed to Jacobi the problem of finding methods for writing numbers that reflect special algebraic properties, i.e., finding periodic representations for algebraic numbers. Continued fractions completely solve this problem for every quadratic irrationality. This is the only known answer, indeed it has not yet been found a method in order to give a periodic representation for every algebraic irrationality of order greater than two. It seems natural to attempt the resolution of the Hermite problem researching some generalization of continued fractions. The first effort in this sense is due to Euler \cite{Euler} in 1749, whose algorithm can provide periodic representations for cubic irrationalities. Successively, the algorithm was modified by Jacobi \cite{Jac} in 1868 and extended to any algebraic irrationalities by Perron \cite{Per} in 1907 (for a complete survey about the Jacobi--Perron algorithm see \cite{Ber}). During the years this generalization and some variations of the continued fractions have been deeply studied. For example Daus \cite{Daus}  developed the Jacobi algorithm for a particular couple of cubic irrationalities, connecting this study with the cubic Pell equation, and Lehemer \cite{Lehmer} examined the convergence of particular periodic expansions. Further developements on the Jacobi--Perron algorithm can be found in \cite{Bre},  \cite{Fer}, \cite{Gar} \cite{Gup1}, \cite{Gup2}, \cite{Ito}, \cite{Sc}.\\
In this paper, we focus only on the Jacobi algorithm and we study periodic representations and approximations for cubic irrationalities. In particualar, introducing a generalization of the R\'edei rational functions \cite{Redei}, for every couple $(\sqrt[3]{d^2},\sqrt[3]{d})$ we provide periodic representations depending on a parameter $z$ which can be any integer. Choosing different values for $z$, it is possible to obtain different periodic expansions and approximations for these irrationalities. Moreover, such representation has the advantage of having a small period making it easy to handle. Indeed, a problem of continued fractions and their generalization is the length of the period, which can be very large. Furthermore, in the case of the cubic irrationalities, periodicity is not guaranteed. As pointed out in \cite{Tamura}, it seems that Jacobi algorithm is not periodic for some couple of cubic irrationalities as, e.g., $(\sqrt[3]{3},\sqrt[3]{9})$ (however it would be possible to find a cubic irrationality $\alpha$ such that the expansion of $(\sqrt[3]{3},\alpha)$ is periodic). Thus, the possibility to obtain a periodic expansion for every couple $(\sqrt[3]{d^2},\sqrt[3]{d})$ appears an important result in order to overcome the problems concerning the periodicity of the Jacobi algorithm.\\
The R\'edei rational functions, generalized in the next section, arise from the development of $(z+\sqrt{d})^n$, where $z$ is an integer and where $d$ is a nonsquare positive integer. One can write
\begin{equation}\label{pow}(z+\sqrt{d})^n=N_n(d,z)+D_n(d,z)\sqrt{d}\ ,\end{equation}
where
$$ N_n(d,z)=\sum_{k=0}^{[n/2]}\binom{n}{2k}d^kz^{n-2k}, \quad D_n(d,z)=\sum_{k=0}^{[n/2]}\binom{n}{2k+1}d^kz^{n-2k-1}.  $$
The R\'edei rational functions $Q_n(d,z)$ are defined by 
\begin{equation}\label{red}Q_n(d,z)=\cfrac{N_n(d,z)}{D_n(d,z)}, \quad \forall n\geq1 \ .\end{equation}
It is well--known their multiplicative property
\begin{equation*}Q_{nm}(d,z)=Q_n(d,Q_m(d,z))\ ,\end{equation*}
for any couple of indexes $n,m$. Thus the R\'edei functions are closed with respect to composition and satisfy the commutative property
$$Q_n(d,Q_m(d,z))=Q_m(d,Q_n(d,z))\ .$$ 
The R\'edei rational functions reveal their utility in several fields of number theory. Given a finite field $\mathbb F_q$, of order $q$, and $\sqrt{d}\not\in\mathbb F_q$, then $Q_n(d,z)$ is a permutation of $\mathbb F_q$ if and only if $(n,q+1)=1$ (\cite{Lidl}, p. 44). Another recent application of these functions provides a way to find a new bound for multiplicative character sums of nonlinear recurring sequences \cite{Gomez}. Moreover, they can be used in order to generate pseudorandom sequences \cite{Topu} and to create a public key cryptographic system \cite{Nob}. In a previous work \cite{bcm} we have seen how R\'edei rational functions can be used in order to generate solutions of the Pell equation in an original way, applying them in a totally new field with respect to the classic ones. Furthermore, in \cite{abcm} we have introduced these functions as convergents of certain periodic continued fractions which always represent square roots. Here we generalize this construction for cubic roots, providing a generalization of the R\'edei rational functions in order to obtain periodic representations for every cubic root.

\section{The Jacobi algorithm and bifurcating continued fractions }
In this section we briefly recall the Jacobi algorithm. It is a generalization of the euclidean algorithm used for constructing the classic continued fractions. In this generalization instead of representing a real number by an integer sequence, a couple of real number is represented by a couple of integer sequences. The algorithm that provides such integer sequences is
\begin{equation} \label{algobiforcanti} \begin{cases} a_n=[x_n] \cr b_n=[y_n] \cr x_{n+1}=\cfrac{1}{y_n-[y_n]} \cr y_{n+1}=\cfrac{x_n-[x_n]}{y_n-[y_n]} \ , \end{cases} \end{equation}
$n=0,1,2,...$, for any couple of real numbers $x=x_0$ and $y=y_0$. We can retrieve $x$ and $y$ from the sequences $(a_n)_{n=0}^{+\infty}$ and $(b_n)_{n=0}^{+\infty}$ using
\begin{equation} \label{algobiforcanti2} \begin{cases} x_n=a_n+\cfrac{y_{n+1}}{x_{n+1}}  \cr y_n=b_n+\cfrac{1}{x_{n+1}}  \end{cases} \end{equation}
$n=0,1,2,...$, indeed by equations (\ref{algobiforcanti}) it follows $$a_n+\cfrac{y_{n+1}}{x_{n+1}}=a_n+\cfrac{\cfrac{x_n-a_n}{y_n-b_n}}{\cfrac{1}{y_n-b_n}}=a_n+x_n-a_n=x_n, \quad \forall n\geq0$$ $$b_n+\cfrac{1}{x_{n+1}}=b_n+\cfrac{1}{\cfrac{1}{y_n-b_n}}=b_n+y_n-b_n=y_n,\quad \forall n\geq0.$$
Therefore, the real numbers $x$ and $y$ are represented by the sequences as follows:
\begin{equation} \label{fcb} x=a_0+\cfrac{b_1+\cfrac{1}{a_2+\cfrac{\left(b_3+\cfrac{1}{\ddots}\right)}{\left(a_3+\cfrac{\ddots}{\ddots}\right)}}}{a_1+\cfrac{b_2+\cfrac{1}{\left(a_3+\cfrac{\ddots}{\ddots}\right)}}{a_2+\cfrac{\left(b_3+\cfrac{1}{\ddots}\right)}{\left(a_3+\cfrac{\ddots}{\ddots}\right)}}} \quad \text{and} \quad y=b_0+\cfrac{1}{a_1+\cfrac{b_2+\cfrac{1}{\left(a_3+\cfrac{\ddots}{\ddots}\right)}}{a_2+\cfrac{\left(b_3+\cfrac{1}{\ddots}\right)}{\left(a_3+\cfrac{\ddots}{\ddots}\right)}}}\end{equation}
We call \emph{ternary} or \emph{bifurcating} continued fraction such couple of objects representing the numbers $x$ and $y$. We call \emph{partial quotients} the integers $a_i$ and $b_i$, for $i=0,1,2,...$ .\\
We can briefly write the bifurcating continued fraction (\ref{fcb}) with the notation $[\{a_0,a_1,a_2,...\},\{b_0,b_1,b_2,...\}]$ and we can introduce the notion of convergent like for the classic continued fraction (for a complete survey of the Jacobi--Perron algorithm see \cite{Ber}). The finite bifurcating continued fraction $$[\{a_0,a_1,...,a_n\},\{b_0,b_1,...,b_n\}]=\left(\cfrac{A_n}{C_n},\cfrac{B_n}{C_n}\right),\quad \forall n\geq0$$ is called $n$\emph{th convergent}, where the integers $A_n,B_n,C_n$ are defined by the following recurrent relations (see, e.g., \cite{Gup1}, \cite{Gup2}) for every $n\geq3$:
\begin{equation} \label{convergenti} \begin{cases} A_n=a_nA_{n-1}+b_nA_{n-2}+A_{n-3} \cr B_n=a_nB_{n-1}+b_nB_{n-2}+B_{n-3} \cr C_n=a_nC_{n-1}+b_nC_{n-2}+C_{n-3}  \end{cases}. \end{equation} 
We can introduce the convergents of a bifurcating continued fraction using a matricial description, like for the classic continued fractions. It is easy to prove by induction that
\begin{equation} \label{matrix-fcb} \begin{pmatrix}  a_0 & 1 & 0 \cr b_0 & 0 & 1 \cr 1 & 0 & 0 \end{pmatrix}... \begin{pmatrix} a_n & 1 & 0 \cr b_n & 0 & 1 \cr 1 & 0 & 0  \end{pmatrix}=\begin{pmatrix} A_n & A_{n-1} & A_{n-2} \cr B_n & B_{n-1} & B_{n-2} \cr C_n & C_{n-1} & C_{n-2}  \end{pmatrix}\end{equation} 
\[ \Updownarrow \] \[\left(\cfrac{A_n}{C_n},\cfrac{B_n}{C_n}\right)=[\{a_0,...,a_n\},\{b_0,...,b_n\}]. \]
\begin{remark}
We observe that the algorithm (\ref{algobiforcanti}) does not provide every bifurcating continued fraction expansion (\ref{fcb}). Indeed, it is easy to prove that using equations (\ref{algobiforcanti}) we always obtain $a_i\geq b_i$ for all $i\geq1$. However, a bifurcating continued fraction can represent a couple of real numbers (i.e., the limit of the convergents exists and it is finite), although it is not obtained starting from the Jacobi algorithm.
\end{remark}
Every periodic bifurcating continued fraction converges to a couple of cubic irrationalities, but it is unproved the viceversa. We do not know if, given any cubic irrationality, another cubic irrationality ever exists such that their bifurcating continued fraction expansion is periodic. Therefore, the Hermite problem is still open for any algebraic irrationalities, except for the quadratic case.\\
The properties of the Jacobi algorithm can be studied by using the characteristic polynomial of the matrices (\ref{matrix-fcb}). For instance, if we consider the purely periodic bifurcating continued fraction $(\alpha,\beta)=[\{\overline{a_0,...,a_n}\},\{\overline{b_0,...,b_n}\}]$, such fraction converges to cubic irrationalities related to the roots of the polynomial
$$\det \begin{pmatrix} A_n-x & A_{n-1} & A_{n-2} \cr B_n & B_{n-1}-x & B_{n-2} \cr C_n & C_{n-1} & C_{n-2}-x \end{pmatrix}=0.$$
Moreover, it is possible to directly study the convergence, considering that in this case we can write  $(\alpha,\beta)=[\{a_0,...,a_n,\alpha\},\{b_0,...,b_n,\beta\}]$ and
$$\alpha=\cfrac{\alpha A_n+\beta A_{n-1}+A_{n-2}}{\alpha C_n+\beta C_{n-1}+C_{n-2}},\quad \beta=\cfrac{\alpha B_n+\beta B_{n-1}+B_{n-2}}{\alpha C_n+\beta C_{n-1}+C_{n-2}}.$$
Similar considerations can be performed in the case of eventually periodic fractions.\\
The difficulties to prove an analogous of the Lagrange theorem (every quadratic irrationalities has periodic continued fraction expansion) arise by the fact that there are no explicit forms for cubic irrationalities. For this reason many different studies of the Jacobi algorithm has been performed. For example, the discussion of bifurcating continued fractions is related to the problem of finding units in cubic fields. Other ways involve the research of bound for the partial quotients and the convergents.\\
The Jacobi algorithm can be also approached studying a transformation of $\mathbb R^2$ into itself (see \cite{Sc}), defined as follows:
$$T(\alpha,\beta)=\left( \cfrac{\beta}{\alpha}-\left[ \cfrac{\beta}{\alpha} \right], \cfrac{1}{\alpha}-\left[\cfrac{1}{\alpha}\right] \right).$$
Using the following auxiliary maps over $\mathbb R^2$, defined by
$$\tau(\alpha,\beta)=\left[\cfrac{1}{\alpha}\right],\quad \eta(\alpha,\beta)=\left[\cfrac{\beta}{\alpha}\right],$$
the partial quotients of the Jacobi algorithm are determined by
\begin{equation} \label {qpT} \begin{cases} a_i=\tau(\alpha_i,\beta_i) \cr b_i=\eta(\alpha_i,\beta_i) \end{cases} \end{equation}
where $(\alpha_i,\beta_i)=T(\alpha_{i-1},\beta_{i-1})$, for $i=0,1,2,...$. Initializing the procedure with $\alpha_0=\cfrac{1}{x}$ and $\beta_0=\cfrac{y}{x}$, for any couple of real numbers $x$ and $y$, the sequence of partial quotients provided by equation (\ref{qpT}) coincides with the sequence provided by the Jacobi algorithm as presented in (\ref{algobiforcanti}). In this way periodicity, convergence and other properties of Jacobi algorithm can be found studying the transformation $T$ which satisfies ergodic properties.\\
In the next section, we propose a different approach, studying the convergence properties of some polynomials and exploiting the fact that they satisfies linear recurrent relations.

\section{Generalized R\'edei rational functions and periodic representations of cubic roots}
As we have seen in the introduction, R\'edei rational functions are strictly connected to square roots. In this paragraph we propose a generalization of the R\'edei polynomials related to every algebraic irrationality and then we will focus on cubic irrationalities, putting in evidence a connection with bifurcating continued fractions and the Hermite problem.\\
Instead of considering the expansion of $(z+\sqrt{d})^n$, we start analyzing the expansion of $(z+\sqrt[e]{d^{e-1}})^n$, where $d$ is not an $e$th power. We observe that in the expansion of $(z+\sqrt[e]{d^{e-1}})^n$ we have coefficients for $\sqrt[e]{d^2},\sqrt[e]{d^3},...,\sqrt[e]{d^{e-1}}$. Thus, we use $e$ polynomials in order to write its expansion:
\begin{equation} \label{ap} (z+\sqrt[e]{d^{e-1}})^n=\mu_n(e,0,d,z)+\mu_n(e,1,d,z)\sqrt[e]{d}+...+\mu_n(e,e-1,d,z)\sqrt[e]{d^{e-1}},  \end{equation}
where
\begin{equation} \label{appoli}  \mu_n=\mu_n(e,k,d,z)=\sum_{h=0}^n\binom{n}{eh-k}d^{(e-1)h-k}z^{n-eh+k}.  \end{equation}
Using the $e \times e $ matrix
\begin{equation} \label{apmatrix} \begin{pmatrix} z & d & 0 & \dots & 0 \cr 0 & z & d & \dots & 0 \cr \vdots & \ddots & \ddots & \ddots & \vdots \cr 0 & \dots & 0 & z & d \cr 1 & 0 & 0 & \dots & z \end{pmatrix},  \end{equation}
whose characteristic polynomial $(x-z)^e-d^{e-1}$ has root of larger modulus $z+d^{(e-1)/e}$, if we define $\mu_n(k)=\mu_n(e,k,d,z)$, where $k=0,1,\ldots,e-1$, we have
\[  \begin{pmatrix} z & d & 0 & \dots & 0 \cr 0 & z & d & \dots & 0 \cr \vdots & \ddots & \ddots & \ddots & \vdots \cr 0 & \dots & 0 & z & d \cr 1 & 0 & 0 & \dots & z \end{pmatrix}^n=\begin{pmatrix} \mu_n(0) & d \mu_n(e-1) & \dots & d \mu_n(1) \cr \mu_n(1) & \ddots & \ddots & \vdots \cr \vdots & \ddots & \ddots & d \mu_n(e-1) \cr \mu_n(e-1) & \dots & \mu_n(1) & \mu_n(0)\end{pmatrix}. \]
\begin{example}
When $e=3$, we have
$$\begin{pmatrix} z & d & 0 \cr 0 & z & d \cr 1 & 0 & z  \end{pmatrix}^n=\begin{pmatrix} \mu_n(0) & d\mu_n(2) & d\mu_n(1) \cr \mu_n(1) & \mu_n(0) & d\mu_n(2) \cr \mu_n(2) & \mu_n(1) & \mu_n(0)  \end{pmatrix}.$$
\end{example}
\begin{remark}
The sequences of polynomials $\mu_n$ are linear recurrent sequences. Indeed, they correspond to the entries of a matrix power, and so they recur with the characteristic polynomial of the resulting power matrix. In particular the sequence $(\mu_n(e,k,d,z))_{n=0}^{\infty}$ recurs with polynomial $(x-z)^e-d^{e-1}$.
\end{remark} 
We can observe the convergence of the polynomials $\mu_n$:
\[\lim_{n\rightarrow\infty} \cfrac{\mu_n(e,k,d,z)}{\mu_n(e,e-1,d,z)}=\sqrt[e]{d^{e-k-1}},\quad k=0,...,e-2. \]
Now, we focus our attention on the cubic case. In the next theorem we prove the convergence of the polynomials $\mu_n$ (since these polynomials will be used in the next paragraph) when $e=3$.
\begin{theorem} \label{mu-3}
Let $d$ be an integer not cube, then
$$\lim_{n\rightarrow\infty}\cfrac{\mu_n(3,0,d,z)}{\mu_n(3,2,d,z)}=\sqrt[3]{d^2}$$
$$\lim_{n\rightarrow\infty}\cfrac{\mu_n(3,1,d,z)}{\mu_n(3,2,d,z)}=\sqrt[3]{d}\ .$$
\end{theorem}
\begin{proof}
The sequence  $\left(\mu_n(3,k,d,z)\right)_{n=0}^{\infty}$ recurs with polynomial $(x-z)^3-d^2$ with real root $\alpha_1=z+\sqrt[3]{d^2}$ of larger modulus than the remaining roots $\alpha_2, \alpha_3$. By the Binet formula
$$\begin{cases} \mu_n(3,0,d,z)=a_1\alpha_1^n+a_2\alpha_2^n+a_3\alpha_3^n \cr \mu_n(3,1,d,z)=b_1\alpha_1^n+b_2\alpha_2^n+b_3\alpha_3^n \cr \mu_n(3,2,d,z)=c_1\alpha_1^n+c_2\alpha_2^n+c_3\alpha_3^n \ .\end{cases} $$
Solving the systems
$$\begin{cases} a_1+a_2+a_3=1 \cr a_1\alpha_1+a_2\alpha_2+a_3\alpha_3=z \cr a_1\alpha_1^2+a_2\alpha_2^2+a_3\alpha_3^2=z^2  \end{cases} \quad \begin{cases} b_1+b_2+b_3=0 \cr b_1\alpha_1+b_2\alpha_2+b_3\alpha_3=0 \cr b_1\alpha_1^2+b_2\alpha_2^2+b_3\alpha_3^2=d  \end{cases} \quad \begin{cases} c_1+c_2+c_3=0 \cr c1\alpha_1+c_2\alpha_2+c_3\alpha_3=1 \cr c_1\alpha_1^2+c_2\alpha_2^2+c_3\alpha_3^2=2z \end{cases}$$
we easily obtain
$$a_1=\cfrac{1}{3},\quad b_1=\cfrac{1}{3\sqrt[3]{d}},\quad c_1=\cfrac{1}{3\sqrt[3]{d^2}}\ .$$
Finally,
$$\cfrac{\mu_n(3,0,d,z)}{\mu_n(3,2,d,z)}=\cfrac{a_1\alpha_1^n+a_2\alpha_2^n+a_3\alpha_3^n}{c_1\alpha_1^n+c_2\alpha_2^n+c_3\alpha_3^n}\rightarrow \cfrac{a_1}{c_1}=\sqrt[3]{d^2}$$
$$\cfrac{\mu_n(3,1,d,z)}{\mu_n(3,2,d,z)}=\cfrac{b_1\alpha_1^n+b_2\alpha_2^n+b_3\alpha_3^n}{c_1\alpha_1^n+c_2\alpha_2^n+c_3\alpha_3^n}\rightarrow \cfrac{b_1}{c_1}=\sqrt[3]{d}\ .$$
\end{proof}
Now, we use the polynomials $\mu_n$ together with bifurcating continued fractions in order to approximate cubic roots. In particular we provide a periodic bifurcating continued fraction expansion for any cubic root, whose convergents are the ratios of these polynomials.\\
First of all, we study bifurcating continued fractions with rational partial quotients.  
\begin{remark}
In \cite{abcm} and \cite{abcm2}, we have studied algebraic properties of continued fractions with rational partial quotients. In this way, it is possible to obtain periodic expansions for quadratic irrationalities more handily. Furthermore, these continued fractions have interesting properties of approximations related to R\'edei rational functions. We proved that they provide, e.g., at the same time Newton and Pad\'e approximations of square roots. Thus, it seems natural to use rational partial quotients with bifurcating continued fractions.
\end{remark}
If we consider a bifurcating continued fraction $$\left[\left\{\cfrac{a_0}{b_0},\cfrac{a_1}{b_1},... \right\},\left\{\cfrac{c_0}{d_0},\cfrac{c_1}{d_1},...\right\}\right]\ ,$$ the sequences $A_n,B_n,C_n$ in (\ref{convergenti}) are rational numbers. Therefore, we can study the recurrence of numerators and denominators of such rational sequences, In the following theorem we provide the result only for the sequence $A_n$, similar results clearly hold for $B_n,C_n$.
\begin{lemma} \label{fcbr-conv} Given
$$\left[\left\{\cfrac{a_0}{b_0},\cfrac{a_1}{b_1},...\right\},\left\{\cfrac{c_0}{d_0},\cfrac{c_1}{d_1},...\right\}\right],$$
let $(A_n)_{n=0}^{\infty},(B_n)_{n=0}^{\infty},(C_n)_{n=0}^{\infty}$ be the sequences such that $$\left[\left\{\cfrac{a_0}{b_0},...,\cfrac{a_n}{b_n}\right\},\left\{\cfrac{c_0}{d_0},...,\cfrac{c_n}{d_n}\right\}\right]=\left(\cfrac{A_n}{C_n},\cfrac{B_n}{C_n}\right)$$ for all $n\geq0$. Then $A_n=\cfrac{s_n}{t_n}$ for every $n\geq0$, where
\[ \begin{cases} s_{-1}=1/d_0, \quad s_0=a_0,\quad s_1=a_0a_1d_1+b_0b_1c_1 \cr
s_n=a_nd_ns_{n-1}+b_nb_{n-1}c_nd_{n-1}s_{n-2}+b_nb_{n-1}b_{n-2}d_nd_{n-1}d_{n-2}s_{n-3},\quad n\geq2 \end{cases} \]
and
\[\begin{cases} t_0=b_0 \cr t_n=b_0\prod_{i=1}^nb_id_i,\quad n\geq1\ . \end{cases}\]
\end{lemma} 
\begin{proof}
We prove the theorem by induction. The verification of the inductive basis is straightforward.\\
Let us suppose the thesis true for all the integers less or equal than $n-1$ and we prove it for $n$. Considering the recurrences (\ref{convergenti}) and the inductive hypothesis, we have
\[ A_n=\cfrac{a_n}{b_n}A_{n-1}+\cfrac{c_n}{d_n}A_{n-2}+A_{n-3}=\cfrac{a_n}{b_n}\cfrac{s_{n-1}}{t_{n-1}}+\cfrac{c_n}{d_n}\cfrac{s_{n-2}}{t_{n-2}}+\cfrac{s_{n-3}}{t_{n-3}}= \] \[=\cfrac{a_ns_{n-1}}{b_nb_0b_1d_1\cdots b_{n-1}d_{n-1}}+\cfrac{c_ns_{n-2}}{d_nb_0b_1d_1\cdots b_{n-2}d_{n-2}}+\cfrac{s_{n-3}}{b_0b_1d_1\cdots b_{n-3}d_{n-3}}=\]
\[=\cfrac{a_nd_ns_{n-1}+b_nb_{n-1}c_nd_{n-1}s_{n-2}+b_nb_{n-1}b_{n-2}d_nd_{n-1}d_{n-2}s_{n-3}}{b_0b_1d_1\cdots b_nd_n}=\cfrac{s_n}{t_n}\ .\] 
\end{proof} 
\begin{remark}
Even if we have posed the condition $s_{-1}=1/d_0$, the sequence $(s_n)_{n=0}^{\infty}$ is an integer sequence, since
$$s_2=a_2d_2s_{1}+b_2b_{1}c_2d_{1}s_{0}+b_2b_{1}b_{0}d_2d_{1}.$$
\end{remark}
It is possible to prove similar results for $(B_n)$ and $(C_n)$:
\[ B_n=\cfrac{s'_n}{t'_n}\quad n\geq0, \]
where $(s'_n)_{n=0}^{\infty}$ recurs as the sequence $(s_n)$, but with initial conditions
$$\begin{cases} s'_0=c_0 \cr s'_1=a_1c_0+b_1d_0 \cr s'_2=b_1b_2c_0c_2+a_1a_2c_0d_2+a_2b_1d_0d_2 \end{cases}$$
and
$$\begin{cases} t'_0=d_0,\quad t'_1=d_0b_1 \cr t'_n=d_0b_1\prod_{i=2}^nb_id_i,\quad n\geq2 \ . \end{cases}$$
Similarly we obtain
\[ C_n=\cfrac{s''_n}{t''_n}, \ \ \ n\geq 0, \]
where $(s''_n)_{n=0}^{\infty}$ recurs as $(s_n)$ with initial conditions 
$$\begin{cases} s''_0=1,\quad s''_1=a_1 \cr s''_2=b_1b_2c_2+a_1a_2d_2  \end{cases}$$
and
$$\begin{cases} t''_0=1,\quad t''_1=b_1 \cr t''_n=b_1\prod_{i=2}^nb_id_i,\quad n\geq2 \ . \end{cases}$$
We can conclude that we have the following expressions for the convergents of a bifurcating continued fraction with rational partial quotients:
\[\cfrac{A_0}{C_0}=\cfrac{s_0}{b_0s''_0},\quad \cfrac{A_n}{C_n}=\cfrac{s_n}{b_0d_1s''_n} \quad \forall n\geq1\]
and
$$ \cfrac{B_n}{C_n}=\cfrac{s'_n}{d_0s''_n}\quad \forall n\geq0, $$
where the sequences $(s_n), (s'_n), (s''_n)$ are integer sequences. Using these sequences we can obtain another matricial representation for the convergents.\\
Now, we study the Hermite problem for cubic irrationalities, observing a connection between the polynomials $\mu_n$ and the bifurcating continued fractions. By using rational partial quotients we can give a periodic expansion for every couple of cubic irrationalities of the kind $(\sqrt[3]{d^2},\sqrt[3]{d})$, whose approximations are provided by the polynomials $\mu_n$. In order to do this, we recall that polynomials $\mu_n=\mu_n(e,k,d,z)$, where now we ever consider $e=3$, have the following matricial representation:
\begin{equation} \label{appmat} \begin{pmatrix} z & d & 0 \cr 0 & z & d \cr 1 & 0 & z \end{pmatrix}^n=\begin{pmatrix} \mu_n(0) & d\mu_n(2) & d\mu_n(1) \cr \mu_n(1) & \mu_n(0) & d\mu_n(2) \cr \mu_n(2) & \mu_n(1) & \mu_n(0)  \end{pmatrix}, \end{equation}
where we write the only dependence from $k$.
\begin{theorem} \label{fcbr} The periodic bifurcating continued fraction
\begin{equation} \label {fcbredei} \left[\left\{z,\cfrac{2z}{d},\overline{\cfrac{3dz}{z^3+d^2},3z,\cfrac{3z}{d}}\right\},\left\{0,-\cfrac{z^2}{d},\overline{-\cfrac{3z^2}{z^3+d^2},-\cfrac{3dz^2}{z^3+d^2},-\cfrac{3z^2}{d}}\right\}\right]  \end{equation}
converges for every integer $z\not=0$ to the couple of irrationals $(\sqrt[3]{d^2},\sqrt[3]{d})$ and its convergents are the couple of rationals $$\left(\cfrac{\mu_n(3,0,d,z)}{\mu_n(3,2,d,z)},\cfrac{\mu_n(3,1,d,z)}{\mu_n(3,2,d,z)}\right),$$ for $\mu_n(e,k,d,z)$ polynomials defined in (\ref{appoli}), $n\geq1$.
\end{theorem}
\begin{proof}
By Theorem \ref{mu-3}, we have only to prove that the convergents of (\ref{fcbredei}) are $$\left(\cfrac{\mu_n(3,0,d,z)}{\mu_n(3,2,d,z)},\cfrac{\mu_n(3,1,d,z)}{\mu_n(3,2,d,z)}\right).$$ 
In this way the periodic bifurcating continued fraction (\ref{fcbredei}) clearly converges to $(\sqrt[3]{d^2},\sqrt[3]{d})$.
For seek of simplicity we specify the only dependence on $k$ for the polynomials $\mu_n$:
\[ \mu_n(3,k,d,z)=\mu_n(k). \]
We will use the representation of the convergents showed in Lemma \ref{fcbr-conv}. We start observing that in this case we have
$$b_0=1,\quad d_0=1,\quad d_1=d$$
and
$$s_0=a_0=z=\mu_1(0),\quad s_1=a_0a_1d_1 + b_0b_1c_1=2dz^2-dz^2=dz^2=d\mu_2(0)$$
$$s'_0=c_0=0=\mu_1(1),\quad s'_1=a_1c_0 + b_1d_0=d=\mu_2(1)$$
$$s''_0=1=\mu_1(2),\quad s''_1=a_1=2z=\mu_2(2)\ .$$
Therefore, for the convergents of (\ref{fcbredei}) we initially have
$$\cfrac{s_0}{b_0s''_0}=\cfrac{\mu_1(0)}{\mu_1(2)}\ ,\quad \cfrac{s'_0}{d_0s''_0}=\cfrac{\mu_1(1)}{\mu_1(2)}$$
$$\cfrac{s_1}{b_0d_1s''_1}=\cfrac{s_1}{ds''_1}=\cfrac{\mu_2(0)}{\mu_2(2)}\ ,\quad \cfrac{s'_1}{d_0s''_1}=\cfrac{\mu_2(1)}{\mu_2(2)}\ .$$
Now, we prove by induction the following relation:
\begin{equation} \label{sn-mun} s_n=d^{\left[\frac{2(n+1)}{3}\right]}(d^2+z^3)^{\left[\frac{2n}{3}\right]}\mu_{n+1}(0),\quad \forall n\geq 2\ .\end{equation}
The inductive basis is straightforward, finding that
$$s_2=d^2(d^2+z^3)\mu_3(0)\ .$$
Now we proceed with the induction, supposing true the thesis for every integer less or equal than $n-1$ and proving the thesis for $n$. Since the period of the fraction is 3, we have to discuss 3 cases for $n$, i.e.,
$$n\equiv 0\mod 3,\quad n\equiv 1\mod 3,\quad n\equiv 2\mod 3.$$
In this proof we only consider the case $n\equiv 0\mod 3$, for the other case the proof is similar. By Lemma \ref{fcbr-conv} we know that
$$s_n=a_nd_ns_{n-1}+b_nb_{n-1}c_nd_{n-1}s_{n-2}+b_nb_{n-1}b_{n-2}d_nd_{n-1}d_{n-2}s_{n-3}.$$
Since $n\equiv 0\mod 3$, we have
$$\begin{cases} a_n=3z,\quad b_n=1,\quad c_n=-3dz^2,\quad d_n=z^3+d^2 \cr a_{n-1}=3dz,\quad b_{n-1}=z^3+d^2,\quad c_{n-1}=-3z^2,\quad d_{n-1}=z^3+d^2 \cr a_{n-2}=3z,\quad b_{n-2}=d,\quad c_{n-2}=-3z^2,\quad d_{n-2}=d \ . \end{cases}$$
Thus, considering the inductive hypothesis, we have
$$s_n=3z(d^2+z^3)d^{\left[\frac{2n}{3}\right]}(d^2+z^3)^{\left[\frac{2n-2}{3}\right]}\mu_{n}(0)$$
$$-3z^2d(d^2+z^3)^2d^{\left[\frac{2(n-1)}{3}\right]}(d^2+z^3)^{\left[\frac{2n-4}{3}\right]}\mu_{n-1}(0)$$
$$+(d^2+z^3)^3d^2d^{\left[\frac{2(n-2)}{3}\right]}(d^2+z^3)^{\left[\frac{2n-6}{3}\right]}\mu_{n-2}(0)$$
i.e.,
$$s_n=3zd^{\left[\frac{2n}{3}\right]}(d^2+z^3)^{\left[\frac{2n+1}{3}\right]}\mu_{n}(0)$$
$$-3z^2d^{\left[\frac{2n+1}{3}\right]}(d^2+z^3)^{\left[\frac{2(n+1)}{3}\right]}\mu_{n-1}(0)$$
$$+(d^2+z^3)d^{\left[\frac{2(n+1)}{3}\right]}(d^2+z^3)^{\left[\frac{2n}{3}\right]}\mu_{n-2}(0).$$
Since $n\equiv 0\mod3$, i.e., $n=3k$, we have the following identities
$$\left[\cfrac{2n}{3}\right]=\left[\cfrac{2(n+1)}{3}\right]=\left[\cfrac{2n+1}{3}\right]$$
indeed,
$$\left[\cfrac{2\cdot3k}{3}\right]=2k,\quad \left[\cfrac{2(3k+1)}{3}\right]=\left[2k+\frac{2}{3}\right]=2k,\quad \left[\cfrac{2\cdot3k+1}{3}\right]=\left[2k+\frac{1}{3}\right]=2k.$$
Therefore, we have
$$s_n=d^{\left[\frac{2(n+1)}{3}\right]}(d^2+z^3)^{\left[\frac{2n}{3}\right]}(3z\mu_{n}(0)-3z^2\mu_{n-1}(0)+(d^2+z^3)\mu_{n-2}(0))$$
and remembering the recurrence relation involving the polynomials $\mu_n$, the equation (\ref{sn-mun}) follows.
It is possible to prove in a similar way the formulas
$$s'_n=d^{\left[\frac{2n-1}{3}\right]}(d^2+z^3)^{\left[\frac{2n}{3}\right]}\mu_{n+1}(1),\quad \forall n\geq2$$
$$ds''_n=d^{\left[\frac{2(n+1)}{3}\right]}(d^2+z^3)^{\left[\frac{2n}{3}\right]}\mu_{n+1}(2),\quad \forall n\geq2.$$
Hence, for the couple of convergents of the fraction (\ref{fcbredei}) we obtain
$$\cfrac{s_n}{b_0d_1s''_n}=\cfrac{s_n}{ds''_n}=\cfrac{\mu_{n+1}(0)}{\mu_{n+1}(2)},\quad \forall n\geq2$$
$$\cfrac{s'_n}{d_0s''_n}=\cfrac{s'_n}{s''_n}=\cfrac{\mu_{n+1}(1)}{\mu_{n+1}(2)},\quad \forall n\geq2$$
and, considering what it has been observed for indexes $n=0,1$, the proof is completed.
\end{proof}
In the previous theorem we have seen an important result about cubic roots. Indeed, we found a periodic representation in the sense of the Hermite problem providing rational approximations for cubic roots related to a generalization of the R\'edei rational functions. It is interesting to underline that the previous representation is valid for every choice of $z$ integer, providing in this way different rational approximations for the same cubic root. Moreover, the period of (\ref{fcbredei}) is really short and it is in general shorter than the period of the bifurcating continued fractions obtained from the Jacobi algorithm. We highlight such  considerations and the differences between our representation and the Jacobi one in the next examples.
\begin{example}
Let us consider the cubic roots $(\sqrt[3]{16},\sqrt[3]{4})$. Choosing, e.g., $z=[\sqrt[3]{16}]=2$, by previous theorem, we have
$$(\sqrt[3]{16},\sqrt[3]{4})=\left[\left\{2,1,\overline{1,6,\cfrac{3}{2}}\right\},\left\{0,-1,\overline{-\cfrac{1}{2},-2,-3}\right\}\right],$$
against the Jacobi algorithm which seems to have non--periodicity. Indeed, evaluating the first 1000 partial quotients of the expansion of $(\sqrt[3]{16},\sqrt[3]{4})$ none periodic patterns appear with the Jacobi algorithm.
\end{example} 
\begin{example}
Let us consider the cubic roots $(\sqrt[3]{25},\sqrt[3]{5})$. In this case the Jacobi algorithm provides a periodic expansion of period 6 and pre--period 7, given by
$$[\{2,1,3,2,1,1,7,\overline{1,1,2,3,1,6}\},\{1,1,1,0,0,0,0,\overline{0,0,0,0,1,2}\}]$$
against our representation of period 3 and pre--period 2, which, choosing, e.g., $z=[\sqrt[3]{5}]=1$, is
$$\left[\left\{1,\cfrac{2}{5},\overline{\cfrac{15}{26},3,\cfrac{3}{5}}\right\},\left\{0,-\cfrac{1}{5},\overline{-\cfrac{3}{26},-\cfrac{15}{26},-\cfrac{3}{5}}\right\}\right].$$
\end{example}
Finally, previous representation allows to retrieve periodic representations and approximations by using linear recurrent sequences for a vast class of cubic irrationalities. Said $(A_n)_{n=0}^{+\infty}, (B_n)_{n=0}^{+\infty}, (C_n)_{n=0}^{+\infty}$ the sequences which determine the convergents of the fraction (\ref{fcbredei}), if we consider the matrix
$$\begin{pmatrix} a_{00} & a_{01} & a_{02} \cr a_{10} & a_{11} & a_{12} \cr a_{20} & a_{21} & a_{22}  \end{pmatrix}\begin{pmatrix} A_n & A_{n-1} & A_{n-2} \cr B_n & B_{n-1} & B_{n-2} \cr C_n & C_{n-1} & C_{n-2} \end{pmatrix}=\begin{pmatrix} \tilde A_n & \tilde A_{n-1} & \tilde A_{n-2} \cr \tilde B_n & \tilde B_{n-1} & \tilde B_{n-2} \cr \tilde C_n & \tilde C_{n-1} & \tilde C_{n-2} \end{pmatrix}$$
it is easy to study the convergence of $\cfrac{\tilde A_n}{\tilde{C_n}}$ e $\cfrac{\tilde B_n}{\tilde C_n}\ .$ In fact, we have
$$\lim_{n\rightarrow\infty}\cfrac{\tilde A_n}{\tilde{C_n}}=\lim_{n\rightarrow\infty}\cfrac{a_{00}A_n+a_{01}B_n+a_{02}C_n}{a_{20}A_n+a_{21}B_n+a_{22}C_n}=\lim_{n\rightarrow\infty}\cfrac{a_{00}\frac{A_n}{C_n}+a_{01}\frac{B_n}{C_n}+a_{02}}{a_{20}\frac{A_n}{C_n}+a_{21}\frac{B_n}{C_n}+a_{22}}\ ,$$
i.e.,
\begin{equation} \label{irr-cub1} \lim_{n\rightarrow\infty}\cfrac{\tilde A_n}{\tilde{C_n}}=\cfrac{a_{00}\sqrt[3]{d^2}+a_{01}\sqrt[3]{d}+a_{02}}{a_{20}\sqrt[3]{d^2}+a_{21}\sqrt[3]{d}+a_{22}} \end{equation}
and similarly
\begin{equation} \label{irr-cub2}\lim_{n\rightarrow\infty}\cfrac{\tilde B_n}{\tilde{C_n}}=\cfrac{a_{10}\sqrt[3]{d^2}+a_{11}\sqrt[3]{d}+a_{12}}{a_{20}\sqrt[3]{d^2}+a_{21}\sqrt[3]{d}+a_{22}}\ .\end{equation}
Therefore, starting from the approximations of the previous theorem we can construct rational approximations, connected to generalized R\'edei polynomials, for all these cubic irrationalities. However, we do not know if these approximations correspond to convergents of some bifurcating continued fraction. To make this happen, it is necessary that the matrix
$$\begin{pmatrix} a_{00} & a_{01} & a_{02} \cr a_{10} & a_{11} & a_{12} \cr a_{20} & a_{21} & a_{22}  \end{pmatrix}$$
is expressed as a product of matrices of type (\ref{matrix-fcb}). The problem of the factorization of any matrix into a product of the kind (\ref{matrix-fcb}) is really hard and we do not study it in this paper. However, we can obtain another interesting result. If we consider the product 
$$A=\begin{pmatrix} a_0 & 1 & 0 \cr b_0 & 0 & 1 \cr 1 & 0 & 0  \end{pmatrix}\cdots\begin{pmatrix} a_3 & 1 & 0 \cr b_3 & 0 & 1 \cr 1 & 0 & 0\end{pmatrix}, $$
the matrix $A$ has first row with entries
$$a_0+a_3+a_0a_1a_2a_3+a_2a_3b_1+a_0a_3b_2+a_0a_1b_3+b_1b_3,\quad 1+a_0a_1a_2+a_2b_1+a_0b_2,\quad a_0a_1+b_1$$
and third row with entries
$$1+a_1a_2a_3+a_3b_2+a_1b_3,\quad a_1a_2+b_2,\quad a_1.$$
If these entries are matched with the entries of a generic matrix
$$\begin{pmatrix} a_{00} & a_{01} & a_{02} \cr a_{10} & a_{11} & a_{12} \cr a_{20} & a_{21} & a_{22}  \end{pmatrix}$$
the system has rational solutions
$$\begin{cases} a_0=1,\quad a_1=a_{22} \cr a_2=\cfrac{1-a_{01}+a_{21}}{a_{22}-a_{02}},\quad a_3=\cfrac{a_{22}-a_{00}a_{22}+a_{02}a_{20}-a_{22}}{a_{02}a_{21}-a_{01}a_{22}} \cr b_0=1,\quad b_1=a_{02}-a_{22} \cr b_2=\cfrac{a_{01}a_{22}-a_{22}-a_{02}a_{21}}{a_{22}-a_{02}},\quad b_3=\cfrac{a_{21}-a_{00}a_{21}+a_{01}a_{20}-a_{01}}{a_{01}a_{22}-a_{02}a_{21}}\ .  \end{cases}$$ 
Using these choices for $a_i,b_i$, $i=0,1,2,3$, the second row of the matrix $A$ can not clearly be any row, but it will be determined by these values. Therefore, we are able to construct a periodic bifurcating continued fraction for any cubic irrationality (\ref{irr-cub1}) paired with another determined cubic irrationality, starting from the fraction (\ref{fcbredei}).

\medskip

\noindent AMS Classification Numbers: 11J68, 11J70

\end{document}